\documentclass[12pt]{amsart}
\usepackage{amsmath}
\usepackage{hyperref}
\usepackage{amssymb,amsmath}
\usepackage{enumerate}
\usepackage{amssymb}
\usepackage{mathrsfs}
\usepackage{frcursive}

\newtheorem{theorem}{Theorem}[section]

\newtheorem{corollary}[theorem]{Corollary}
\newtheorem{lemma}[theorem]{Lemma}

\theoremstyle{definition}

\theoremstyle{remark}
\newtheorem{remark}[theorem]{Remark}

\newcommand{\R}{{\mathbb R}}

\newcommand{\N}{{\mathbb N}}

\newcommand{\C}{{\mathbb C}}

\newcommand{\Z}{{\mathbb Z}}

\newcommand{\BC}{{\mathbb B}{\mathbb C}}

\renewcommand{\bar}[1]{{\overline{#1}}}

\newcommand{\norm}[1]{\| #1 \|}
\newcommand{\Arg}{{\operatorname{Arg}}}
\renewcommand{\Re}{{\operatorname{Re}}}
\renewcommand{\Im}{{\operatorname{Im}}}

\numberwithin{equation}{section}

\usepackage{tikz}
\usetikzlibrary{matrix}

\newcommand{\nc}{\newcommand}
\nc{\G}{{\Gamma}}
\nc{\BG}{{\mathbb G}}
\nc{\V}{{\mathbb V}}
\nc{\E}{{\mathbb E}}
\nc{\F}{{\mathbb F}}
\nc{\BR}{{\mathbb R}}
\nc{\BZ}{{\mathbb Z}}
\nc{\BP}{{\mathbb P}}
\nc{\BA}{{\mathbb A}}
\nc{\BM}{{\mathbb M}}
\nc{\BN}{{\mathbb N}}
\nc{\BT}{{\mathbb T}}
\nc{\fH}{{\mathfrak{H}}}
\nc{\vp}{{\varepsilon}}\nc{\dpar}{{\partial}}\nc{\al}{{\alpha}}
\nc{\PSL}{{\mbox{PSL}_2(\BR)}}
\nc{\PS}{{\mbox{PSL}_2(\BZ)}}
\nc{\SL}{{\mbox{SL}_2(\BZ)}}
\nc{\SLL}{{\mbox{SL}_2(\BZ[i])}}
 \nc{\CL}{{\mbox{PSL}_2(\BC/n\BC)}}
\newtheorem{thm}{Theorem}[section]

\theoremstyle{remark}
\newtheorem{rem}{Remark}[section]

\newcommand{\floor}[1]{{\left\lfloor #1 \right\rfloor}}

\DeclareMathOperator{\lcm}{lcm}
\RequirePackage[all]{xy}

\begin{document}
\title{ Arithmetic and Analysis of the series $\displaystyle {  \sum_{n=1}^{\infty} \frac{1}{n}  \sin  \frac{x}{n} }$  }
\author{Ahmed Sebbar}
\address{Chapman University\\
  One University Drive,\\
  Orange CA 92866, USA}%
\email{sebbar@chapman.edu}
\address{Universit\'e de Bordeaux\\ IMB, UMR 5251, F-33405 Talence, France}
 \email{ahmed.sebbar@math.u-bordeaux.fr}
\author{Roger Gay}

\address{Universit\'e de Bordeaux\\ IMB, UMR 5251, F-33405 Talence, France} 
\email{roger.gay@math.cnrs.fr }

%\date{\today}
%\blfootnote{Data availability statement : the author confirms that the data supporting the findings of this study are available within the article [and/or] its supplementary materials. Other declaration: not applicable}

\maketitle

\vspace*{0.6cm}

\hspace{6cm}
{{\it To the memory of our friend Carlos Berenstein}}

\vspace*{0.6cm}

\begin{abstract}
  In this paper we connect  a celebrated theorem  of Nyman and Beurling on the equivalence between the Riemann hypothesis  and the density of some 
   functional space in $ L^2(0, 1)$ to a trigonometric series considered first by Hardy and Littlewood (see \eqref{main}). We highlight some of its curious analytical and arithmetical properties.
    \end{abstract}

\section{introduction}
The main purpose of this work is to bring to light a new relationship between two facets of Riemann's zeta function: On the one hand a functional analysis approach to the Riemann hypothesis due to Nymann and Beurling, and on the other hand a trigonometric series first studied by Hardy and Littlewood \cite{Hardy}, and then followed by Flett \cite{Flett}, Segal \cite{Segal} and Delange \cite{Delange}. The trigonometric series in question is  
\begin{equation}\label{Fundamental}
 \mathfrak f (x)= \sum_{n=1}^{\infty}  \frac{1}{n} \sin \frac{x}{n}.
\end{equation}
It differs from the finite sum  $\displaystyle \sum_{n\leq x} \frac{1}{n} \sin \frac{x}{n} $, as $x$ tends to $\infty$, by
\[ \sum_{n>1}  \frac{1}{n} \sin \frac{x}{n}= O\left(\sum_{n>x} \frac{x}{n^2} \right)= O(1).\]
Hardy and Littlewood proved \cite{Hardy} that, as $x$ tends to $\infty$,
\[ \mathfrak f (x)= O\left(   (\log x)^{\frac{3}{4}}         (\log \log x)^           {\frac{3}{4}+\epsilon}                \right)    \]
and that
\[\mathfrak f (x)=  \Omega\left( (\log \log x)^           {\frac{3}{4}}    \right) \]
from the fact that for $x\geq 5 $, the number of $n\leq x$ whose prime divisors are 
equivalent to $1$ modulo $4$ is $\displaystyle C \frac{x}{(\log x)^{\frac{1}{2}}}$, where $C$ is a constant. Delange \cite{Delange} showed that $\mathfrak f (x)$ is not bounded on the real line only from the following result on the reciprocals of primes in arithmetic progressions
\[\displaystyle \sum_{\substack{ p\,{\rm prime}, \\ p\equiv 1({\rm mod}\,4) }}\frac{1}{p}= \infty\]
and obtained the $\Omega$-result of Hardy and Littlewood  just because 
\[\displaystyle \sum_{\substack{p\,{\rm prime}\leq x, \\p\equiv 1({\rm mod}\,4)} }\frac{1}{p}= \frac{1}{2} \log \log x+ c+ o(1).\]

This trigonometric series, despite its simplicity, has many similarities with the Riemann zeta function \cite{Flett} and deep relation to the divisor 
functions through the sawtooth function
\begin{equation}\label{Beurling1}
\{t \}= -\frac{1}{\pi} \sum_{m=1}^{\infty} \frac{\sin 2m\pi t}{m}= \left\{ \begin{matrix}
t-\floor{t}-\frac{1}{2}& {\rm if}\quad  t\neq \floor{t}\\
\\
0& {\rm if}\quad  t= \floor{t}.
\end{matrix}
\right.
\end{equation}
For $s\in \C$ we define 
\[\sigma^s(n)= \sum_{d\vert n} d^s,\quad \sigma_s(n)= \sum_{d\vert n} d^{-s}\]
so that $\displaystyle n^s \sigma_s(n)=  \sigma^s(n)$. For example if we  define
\[S_1(x)= \sum_{n\leq x}   \sigma_1(n), \quad    S^1(x)= \sum_{n\leq x}   \sigma^1(n)  \]
and 
\[\rho(x)= \sum_{n\leq x} \frac{1}{n} \{ \frac{x}{n}\}= \sum_{n\leq x} \frac{1}{n}\left(  \frac{x}{n}-\floor{\frac{x}{n}}- \frac{1}{2}\right)\]
then the divisors and the fractional parts functions are related by
\begin{align*}
S_1(x)&= \sum_{n\leq x} \frac{1}{n} \floor{\frac{x}{n}}= x\sum_{n\leq x} \frac{1}{n^2}-\rho(x)\\
{}&= \frac{\pi^2}{2}x-\frac{1}{2}\log x- \rho(x)+ O(1).
\end{align*}
Similarly \cite{Walfisz} (p.70):
\[S^1(x)= \frac{\pi^2}{12}x^2- x\rho(x)+ O(x).\]
We will see  (\eqref{Delange} with $f(2\pi x)= \sin x  $) an integral representation of the partial sums of  $ \mathfrak f (x)$, using  the sawtooth function.
\section{Nyman-Beurling criterion for the Riemann hypothesis}

\subsubsection{Nyman-Beurling theorem} For $x>0$, let $\rho(x)$ be the fractional  part of $x$ so that $\rho(x)=x-\floor{x}$. 
To each $0 < \theta \leq 1 $ we associate the function $\displaystyle \rho_{\theta}(x) = \rho(\frac{\theta}{x})$. Then $0\leq   \rho_{\theta}(x) \leq 1 $ and $\displaystyle \rho_{\theta}(x)=  \frac{\theta}{x} $
  if $\theta <x  $. We introduce, as in  \cite{Beurling}, \cite{Donoghue}, \cite{Duarte}, \cite{Balazard1},  \cite{Balazard2},  \cite{Nikolski1}, \cite{Vasyunin} and the more recent book 
  \cite{Nikolski2}
  \[{\mathcal M}= \left \{f, f(x)= \sum_{n=1}^N a_n  \rho(\frac{\theta_n}{x}), \, a_n\in \R, \theta_n \in (0,1],\,  \sum_{n=1}^N a_n \theta_n= 0,\, N\geq 1 \right\}. \]
Each function in ${\mathcal M}$ has at most a countable set of points of discontinuity, and is identically zero for $x>0$.
\begin{thm}[Nyman-Beurling]
Let $1< p\leq \infty $. The subspace $ \mathcal M$ is dense in the Banach space $L^p(0, 1)  $ if and only if the Riemann zeta function $\zeta(s)$ has no zero in the right half plane $\displaystyle \Re s> \frac{1}{p}.       $
\end{thm}
The fundamental relations in the proof of this theorem are
\begin{equation}\label{Beurling11}
\int_0^1  \rho(\frac{\theta}{x}) x^{s-1}\,dx= -\frac{\theta}{1-s}-\theta^s \frac{\zeta(s)}{s},\quad \Re s>1,
\end{equation}
which is just a variant of the classical representation
\begin{equation}\label{classical0}
\zeta(s)= \frac{s}{s-1}-s\int_0^{\infty} \frac{u-\floor{u}}{(u+1)^{s+1}}\,du. 
\end{equation}
It follows  from \eqref{Beurling11} that  for $f(x)\in {\mathcal M} $
\[\int_0^1 f(x) x^{s-1}\,dx= -\frac{\zeta(s)}{s}  \sum_{k=1}^N a_k \theta_k^s.\]

The study of the function $\mathfrak{f}(x) $ is  intimately linked to that of following function 
\begin{equation*}
\{t \}= \left\{ \begin{matrix}
t-\floor{t}-\frac{1}{2}& {\rm if}\quad  t\neq \floor{t}\\
\\
0& {\rm if}\quad  t= \floor{t}.
\end{matrix}
\right.
\end{equation*}
We have the formal Fourier series expansion \cite{Davenport1},  \cite{Davenport2}
\begin{equation}\label{Beurling2}
\sum_{n=1}^{\infty} \frac{a_n}{n}\{n\theta\}= -\frac{1}{\pi} \sum_{n=1}^{\infty} \frac{A_n}{n}\{\sin 2\pi n\theta\}
\end{equation}
where
\[A_n= \sum_{d\mid n} a_d. \]
Davenport considered the cases of
\[ a_n= \mu(n);\;  a_n= \lambda(n);\; a_n= \Lambda(n);\; a_{n^2}= \mu(n), a_n= 0,\, n\neq m^2.\] 
These arithmetical functions have their usual number-theoretic meanings. For example if $\omega(n)$ is the number of distinct prime factors
 of $n$ or, in other terms, $\displaystyle \omega(n)= \sum_{b\mid n} 1$ and $\omega(1)= 0$, then the M\"obius function $\mu(n)$ is defined by
 \[
 \mu(n)=\left \{ \begin{matrix}
0 & {\rm if}\, n\,{\rm is \,divisible\, by \,a \,perfect\, square } >1\\
\\
(-1)^{\omega(n)}& {\rm otherwise}
\end{matrix}
 \right.
 \]
and the Von Mangoldt function $\Lambda(n) $ is defined by
 \[
 \Lambda(n)=\left \{ \begin{matrix}
\log p& {\rm if}\, n= p^{\alpha}\,{\rm for \;a\; prime}\; p\;{\rm  \;and\; some }\,\alpha\in \N\\
\\
0 & {\rm otherwise}
\end{matrix}
 \right. 
 \]
In the case of the M\"obius function $a_n= \mu (n)$, Davenport uses
Vinogradov's method, a refinement of Weyl's method on estimating trigonometric sums, to prove that for any fixed $h$
\begin{equation}\label{Davenport}
\sum_{n\leq y} \mu(n) e^{2i\pi n x}= O\left(y(\log y)^{-h} \right)
\end{equation}
uniformly in $x\in \R/\Z$. The implied constants are not effective. 
There have been several results justifying \eqref{Beurling2}  for other particular sequences $(a_n)$. The most general problem is considered in \cite{Hartman}.

It should be noted that the Davenport or Hardy Littlewood estimates admit a common analysis. For the convenience of the reader we gather together a few classical results  on exponential sums. Let $\bf I$ be an interval of length at most $N\geq 1$ and let $f: {\bf I} \to \R $ be a smooth function satisfying the estimates $x\in {\bf I}, 2\leq N\ll T,  j\geq 1$
\[    \vert f^{(j)} (x)\vert= {\rm exp}\left(O(j^2) \right) \frac{T}{N^j}          \]
then with $f(x)= e^x$
\begin{enumerate}

\item Van der Corput estimate: For any natural number $k\geq 2$, we have
\begin{equation}\label{Corput}
\displaystyle \frac{1}{N} \sum_{n\in {\bf I}} e(f(n))= O\left( \frac{T}{N^k}^{\frac{1}{2^k-2}} \log^{\frac{1}{2}}(2+T)\right)
\end{equation}
\item Vinogradov estimate: For some absolute constant $c>0$.
\begin{equation}\label{Vinogradov}
\displaystyle \frac{1}{N} \sum_{n\in {\bf I}} e(f(n)) \ll N^{-\frac{c}{k^2}}.
\end{equation}

\end{enumerate}

\subsubsection{The functions $\floor{x},\;\rho(x) $ and $\{x\}$}
The Hardy-Littlewood-Flett function $ { \mathfrak f (x)}$ is related, in many ways, to the three functions $\floor{x},\;\rho(x) $ and $\{x\}$. The  floor function $  \floor{x}  $ is related to the divisor function $d(n)= 1\star1 (n) $, the multiplicative square convolution product of the constant function $1$, through the Dirichlet hyperbola method. More generally if $g, h$ are two multiplicative functions and $f=g\star h $. The Dirichlet hyperbola method is just the evaluation of a sum in two different ways:
\begin{eqnarray*}
\displaystyle\sum_{n\leq x}f(n)&= \displaystyle\sum_{n\leq x} \,\sum_{ab=n} g(a) h(b)= \sum_{a\leq \sqrt{x}}\; \sum_{b\leq \frac{x}{a}} g(a) h(b)+ \sum_{b\leq \sqrt{x}}\; \sum_{a\leq \frac{x}{a}} g(a) h(b)\\
{}&\displaystyle-\sum_{a\leq \sqrt{x}}\; \sum_{b\leq \sqrt{x}} g(a) h(b).
\end{eqnarray*}

If $g=h$, then
\[\displaystyle\sum_{n\leq x}f(n)= 2\sum_{a\leq \sqrt{x}} \sum_{b\leq \frac{x}{a}} g(a) h(b) - \left(\sum_{a\leq \sqrt{x}}g(a)  \right)^2. \]
As an application we have the estimate \cite{Titch} (p. 262) for the divisor function $d= 1\star 1$:
\[ d(x)= x\log x +(2\gamma-1)x+ O(x^{\frac{1}{2}}).\]
The importance of the functions  $\{x \}$ and $\rho(x)$  lies in the integral representations of the Riemann zeta-function:
\[\zeta(s)= -s\int_0^{\infty}  \frac{\{x\}-\frac{1}{2}}{x^{s+1}}\;dx= -s\int_0^{\infty}  \frac{\rho(x)}{x^{s+1}}\;dx\]
valid for $-1<\Re s< 0 $. Making the change of variable $\displaystyle x= \frac{1}{u} $ and applying Mellin inversion formula gives
\begin{equation*}
\rho(\frac{1}{u})= -\frac{1}{2i\pi} \int_{c-i\infty}^{c+i\infty} \frac{\zeta(s)}{s} u^{-s} \,ds.
\end{equation*}
For later use, we give some details on the case considered by Davenport in \eqref{Beurling2}. From \eqref{Davenport} we obtain for $ -1<c<0$
\[\sum_{n=1}^{\infty} \frac{\mu(n)}{n} \rho(nx)= -\frac{1}{2i\pi}\int_{c-i\infty}^{c+i\infty}\frac{\zeta(s)}{s\zeta(1-s)} x^s\,ds.        \]
By the functional equation of the Riemann $\zeta$-function and the functional equation of the $\Gamma$-function we obtain for $ 0<a<1$
\[ \sum_{n=1}^{\infty} \frac{\mu(n)}{n} \rho(nx)=    -\frac{1}{2i\pi^2}\int_{a-i\infty}^{a+i\infty}\Gamma(s)\sin(\frac{1}{2}\pi s)  (2\pi x)^{-s}\,ds=   
-\frac{1}{\pi}\sin(2\pi x).  \]
Using the classical equivalent formulation of the Prime Number Theorem that $\displaystyle \sum_{n=1}^{\infty}   \frac{\mu(n)}{n}= 0   $ we obtain Davenport's relation
\begin{equation}\label{Beurling3}
 \sum_{n=1}^{\infty} \frac{\mu(n)}{n} \{nx\}= -\frac{1}{\pi}\sin(2\pi x)
\end{equation}
where the convergence is uniform by Davenport estimate \eqref{Davenport}.
We will need two important properties of the function $\{x \}$:\\
Kubert identity:
\begin{equation}\label{Franel1}
\sum_{l\,{\rm mod}\,m}\left\{x+\frac{l}{m}\right\}= \left\{m x\right \}
\end{equation}
Franel formula:
\begin{equation}\label{Franel2}
\int_0^1  \left\{a x\right \} \left\{b x\right \}\,dx= \frac{\lcm( a,b)}{12 ab}.
\end{equation}
Kubert identity and  Franel's formula are interesting features shared by many functions. Let $B_r(x)$ be the Bernoulli polynomial defined by
\[\frac{te^{tx}}{e^t-1}=  \sum_{r=0}^{\infty} B_r(x)t^r,\quad \vert t\vert <2\pi, \]
so that
\[  B_1(x)= x-\frac{1}{2},\quad   2!B_2(x)= x^2-x+\frac{1}{6}, \cdots \]
If $r\geq 2$ is even then for $ 0\leq x\leq 1        $
\[ (2\pi)^rB_r(x)= (-1)^{\frac{1}{2}- 1} \sum_{l=1} ^{\infty}    \frac{2 \cos (2l\pi x)}{l^r}       \]
with absolute convergence of the series. The Hurwitz zeta function $\zeta(s,x)  $ is defined for $\Re s>1 $ by
\[\zeta(s,x)= \sum_{n=0}^{\infty} \frac{1}{(x+n)^s}.  \]
Then $B_r(x)$ and $ \zeta(s,x)$ both satisfy the functional equation \cite{Mordell}
\[f(x)+f(x+\frac{1}{k}) +\cdots+f(x+\frac{k-1}{k})= f^{(k)} f(kx),\]
where $ f^{(k)}= k^{1-n}$ if $f(x)= B_n(x)$ and $ f^{(k)}= k^{s}$ if $f(x)=  \zeta(s,x)$. Furthermore, if $ a, b$ denote arbitrary positive integers and $ (a, b)= \gcd(a, b), [a, b]= \lcm(a, b) $ the greatest common divisor and least common multiple respectively of $a$ and $b$,  then \cite{Mordell}:
\[\int_0^1 B_r({ax})  B_r({bx}) \,dx= (-1)^{r-1} \frac{B_{2r}}{ (2r)!} \left(\frac{(a,b)}{[a,b]} \right)^r     \]
and for $\Re s> \frac{1}{2} $
\[  \int_0^1 \zeta(1-s, {ax})   \zeta(1-s, {bx}) \,dx=    \frac{2\Gamma^2(s)\zeta(2s)}{ (2\pi)^{2s}} \left(\frac{(a,b)}{[a,b]} \right)^s.   \]
Similarly to \eqref{classical0} we have
\[\zeta(s,w)= \frac{1}{(s-1)w^{s-1}}+\frac{1}{w^s}-s \int_0^{\infty} \frac{u-\floor{u}}{(u+w)^{s+1}}\,du, \]
and the function $\displaystyle \zeta(s,w)- \frac{1}{(s-1)w^{s-1}}$ is analytic in $\{\Re s>0 \} $.
In the next section we use two summation formulas.

 If $F$ is an antiderivative of $f$, then, formally \cite{Bod1}
\begin{equation}\label{Franel3}
 \int_0^1\rho(\frac{\theta}{t}) f(t)dt= \theta  \int_0^{1} \frac{f(t)}{t}\,dt - \sum_{n=1}^{\infty}n\left( F(\frac{\theta}{n})-F(\frac{\theta}{n+1})\right)
\end{equation}
and 
if $\mu$ is the M\"obius function and if $0< \theta,\,x \leq 1 $ , we have, pointwise \cite{Bod}
\[ \sum_{n=1}^{\infty} \mu(n)\left\{\rho\left(\frac{\theta}{nx}\right)-\frac{1}{n}\rho\left(\frac{\theta}{x}\right)  \right\}= -\chi_{]0, \theta]}(x).\]

\section{From Beurling's theorem to Hardy-Littlewood-Flett function $\mathfrak{f} (x)$}
\subsection{The emergence of Franel integral type}
To show that the constant function $1 \in \bar{ {\mathcal M}}$ one has, as in \cite{Bod},  to minimize the norms in $L^2([0,\,1])     $
\begin{equation}
\norm {1+ \sum_{j=1}^N a_j \rho \left( \frac{\alpha_j}{\cdot} \right) }
\end{equation}
which brings back to the evaluation of integrals of Franel type, computed in \cite{Bod}:
\[J(\beta)= \int_0^1  \rho \left( \frac{1}{x} \right) \rho \left( \frac{\beta}{x} \right)\, dx, \quad \beta\in [0,1]. \]
To show that the  function $\sin x \in \bar{ {\mathcal M}}$ one has, this time,  to minimize the norms 
\begin{equation}
\norm {\sin x+ \sum_{j=1}^N a_j \rho \left( \frac{\alpha_j}{\cdot} \right) }
\end{equation}

Using \eqref{Beurling3} the  minimization problem reduces to evaluation of the scalar products in $\displaystyle L^2\left(0,\,1 \right) $ giving the Fourier sine series of the function $  \{ \frac{\theta}{x}\}        $, that is 
\[
a_n= (\{\frac{\theta}{x}\}{\big\vert}\sqrt{2}\sin(n\pi x))=\sqrt{2}\int_0^1\{\frac{\theta}{x}\}\sin(n \pi x) dx= -\pi\sqrt{2}\sum_{j \geq 1}\frac{\mu_j}{j}\int_0^1\{\frac{\theta}{x} \}\{\frac{jn x}{2}\} dx
\]
and then to the evaluation of $\displaystyle \int_0^1 \{\frac{a}{x}\} \{bx\}dx$, another kind of  integrals of Franel type. We compute these integrals in the case $a=m,\,b=n $, $m$ and  $n$  being integers.
\subsection{The second kind of Franel type integrals $I_{n,m}=\int_0^1\{nx\}\{\frac{m}{x}\}dx,\;n,m\in \N^*$}
The values of the integrals $I_{n,m}$ are given by the following
\begin{thm}\label{extendedFranel}
For positive integers $m,n$, the modified Franel integrals are given by
 \begin{equation*}\begin{split}
 I_{n,m}&= \displaystyle \frac{n}{m}+m\log m+m(n-1)\log(mn)-m(\log((n-1)!))\\
&- \displaystyle \frac{n(n-1)}{2}-\frac{nm^2}{2}(\zeta(2)-\sum_{1\leq j\leq m}(1-\frac{m}{j}))+\sum_{1\leq k\leq n-1,mn\geq jk}(1-\frac{mn}{jk}).
\end{split}\end{equation*}
\end{thm}
Let us first give few examples:
 \begin{align*}
\displaystyle I_{(2,1)}= &\displaystyle  \frac{5}{2}-\log(2)-\zeta(2);& \displaystyle &I_{(3,1)}=\displaystyle \frac{25}{6}+\log(2)-2\log(3)-\frac{3}{2}\zeta(2)\\
%\displaystyle (I_{3,1)}&=&\displaystyle \frac{25}{6}+\log(2)-2\log(3)-\frac{3}{2}\zeta(2)\quad 
 \displaystyle I_{(4,1)}= &\displaystyle \frac{35}{6}-5\log(2)+\log(3)-2\zeta(2);& \displaystyle &I_{(5,1)}= \displaystyle  \frac{35}{6}-5\log(2)+\log(3)-2\zeta(2)    \\
%\displaystyle I_{(5,1)}&=&  \displaystyle  \frac{35}{6}-5\log(2)+\log(3)-2\zeta(2)    \\
\displaystyle I_{(1,2)}=& \displaystyle \frac{7}{2}-2\zeta(2);& \displaystyle &I_{(1,3)}= \displaystyle \frac{61}{8}-\frac{9}{2}\zeta(2)  \\
%\displaystyle I_{(1,3)}&=& \displaystyle       \frac{61}{8}-\frac{9}{2}\zeta(2)  
\displaystyle I_{(1,4)}= & \displaystyle  \frac{5989}{288}-\frac{25}{2}\zeta(2);& \displaystyle& I_{(2,2)}= \displaystyle \frac{49}{6}-2\log(2)-4\zeta(2)\\
 %\displaystyle I_{(2,2)}&=&  \displaystyle \frac{49}{6}-2\log(2)-4\zeta(2)\\
 \displaystyle I_{(2,3)}=& \displaystyle   \frac{171}{10}-3\log(2)-9\zeta(2);&  \displaystyle &I_{(2,4)}= \displaystyle \frac{18469}{630}-4\log(2)-16\zeta(2)\\ 
%\displaystyle I_{(2,4)}&=& \displaystyle \frac{18469}{630}-4\log(2)-16\zeta(2)\\ 
\displaystyle I_{(2,5)}=& \displaystyle    \frac{15059}{336}-5\log(2)-25\zeta(2);&\displaystyle& I_{(3,2)}= \displaystyle  \frac{196}{15}+2\log(2)-4\log(3)-6\zeta(2) 
\end{align*}
We observe that in all these examples the factor
  $\displaystyle  \zeta(2)= \frac{\pi^2}{6}$ is present.

  For the proof we consider the two functions defined on $]0,+\infty[$
  \[f(x)= f_n(x)= x\chi_{\lbrack0,1\rbrack}(x)\{nx\},\quad g(x)=\{x\}\chi_{\lbrack1,+ \infty \rbrack}(x)\] and their multiplicative convolution
  \[
  (f \star g)(a)= \int_0^{+ \infty}f(x)g(\frac{a}{x})\frac{dx}{x}, \quad  (f \star g)(m)= I_{n,m}.
  \]
 We split the computations in several steps.  A natural method is to use  first  the Mellin transform with its property ${\mathcal M}(f \star g) (s) = {\mathcal M}(f)(s){\mathcal M}(g) (s) $, followed by an inversion.   The main idea is the decomposition formula \eqref{Franel3}, valid if 
 $\displaystyle \int_0^1 \frac{\vert f(x)\vert }{x}\, dx$ is finite:
 \[\int_0^1 \rho_{\theta}(x) f(x) dx=  \sum_{n=1}^{\infty} \int_{\frac{\theta}{n+1}}^{\frac{\theta}{n}}(\frac{\theta}{x}-n)f(x) dx  +\int_{\theta}^1 \frac{\theta}{x} f(x) dx, \]
or in an generalized function form,
\[\rho_{\theta}(x)= \sum_{n=1}^{\infty} (\frac{\theta}{x}-n) \chi_{[\frac{\theta}{n+1}, \frac{\theta}{n}]}(x)+ \frac{\theta}{x} \chi_{[\theta, 1]} \]
where $\chi_B$ denotes the characteristic function of the set $B$.
\vspace*{0.2cm}
  \subsubsection{Computations of different integrals}

  \begin{enumerate}
  \item{Computation of $ \displaystyle F(s) = M(f)(s)$}
  For $\sigma = \Re\,s >-2$ we have
  \begin{equation*}\begin{split}
  F(s)&=\int_0^1\{nx\}x^sdx \\
  &=\sum_{0\leq k\leq n-1}\int_{\frac{k}{n}}^{\frac{k+1}{n}}(nx-k)x^sdx\\
  &=n\int_0^1x^{s+1}dx-\sum_{1\leq k\leq n-1}k\int_{\frac{k}{n}}^{\frac{k+1}{n}}x^sdx\\
  &=\frac{n}{s+2}-\frac{1}{(s+1)n^{s+1}} \sum_{1\leq k\leq n-1}k((k+1)^{s+1}-k^{s+1})\\
  &=\frac{n}{s+2}-\frac{1}{(s+1)n^{s+1}}\{n^{s+2}-(1+2^{s+1}+3^{s+1}+\cdots n^{s+1})\}\\
 \end{split}\end{equation*}
 \item{Computation of $ \displaystyle G(s)=M(g)(s)$}
 For $-2<\sigma= {\Re}s<-1$ we have
 \begin{equation*}\begin{split}
 G(s)&=\int_1^{+\infty} {\{x\}}^{s-1}dx\\
 &=\sum_{k\geq1}\int_k^{k+1} (x-k)x^{s-1}dx\\
 &=\int_1^{+\infty}x^sdx-\sum_{k\geq1}k\int_k^{k+1}x^{s-1}dx\\
 &=\int_1^{+\infty}x^sdx-\frac{1}{s}\sum_{k\geq1}k((k+1)^s-k^s)\\
 &=\frac{1}{s+1}-\frac{\zeta(-s)}{s}\quad \sigma<-1\\
 \end{split}\end{equation*}
 Hence for  $-2 < c<-1$ we can write
 \begin{equation*}\begin{split}
 I_{n,m}=&\\
  \frac{1}{2i\pi}\int_{c-i\infty}^{c+i\infty}&\left(\frac{1}{s+1}-\frac{\zeta(-s)}{s}\right)\left(\frac{n}{s+2}-\frac{1}{(s+1)n^{s+1}}(n^{s+1}-(1+2^{s+1}+3^{s+1}+\cdots(n-1)^{s+1})\right)\frac{ds}{m^s}.
 \end{split}\end{equation*}
  \noindent and, by changing $s$ to $-s$, we get for $1<c<2$
 \begin{equation*}\begin{split} 
 I_{n,m}=&\\
  \frac{1}{2i\pi}\int_{c-i\infty}^{c+i\infty}&\left(\frac{1}{1-s}+\frac{\zeta(s)}{s}\right)\left(\frac{n}{2-s}-\frac{1}{(1-s)n^{1-s}}(n^{1-s}-(1+2^{1-s}+3^{1-s}+\cdots(n-1)^{1-s})\right)\frac{ds}{m^{-s}}.
 \end{split}\end{equation*}
 By expanding we find:
 \begin{equation*}\begin{split}
 I_{n,m}&=\frac{1}{2i\pi}\int_{c-i\infty}^{c+i\infty}\frac{n.m^s}{(1-s)(2-s)}ds\\
 &-\frac{1}{2i\pi}\int_{c-i\infty}^{c+\infty}\frac{m^s}{(1-s)^2n^{1-s}}(n^{1-s}-(1+2^{1-s}+3^{1-s}+\cdots(n-1)^{1-s}))ds\\
 &+\frac{1}{2i\pi}\int_{c-i\infty}^{c+i\infty}\frac{\zeta(s)}{s}(\frac{n}{2-s}-\frac{1}{(1-s)n^{1-s}}(n^{1-s}-(1+2^{1-s}+\cdots+(n-1)^{1-s}))m^sds\\
 \end{split}\end{equation*}
 We  treat the last integral by expanding the $ \zeta $ function in Dirichlet series. We will treat each type of integrals appearing separately. Then we proceed to the necessary groupings in order to conclude. \\
 In the following we  write $\displaystyle \int_{(c)} $ \ instead of $\displaystyle \int_ {c-i \infty} ^ {c+ i \infty} $,  with $ 1 <c <2 $. 
 \item{Computation of   $ \displaystyle \frac{n}{2i\pi}\int_{(c)} \frac{m^s}{(1-s)(2-s)}ds$}. We set $\displaystyle  f(x)=-\frac{1}{x}$  for $0<x\leq1$ and $f(x)= \displaystyle -\frac{1}{x^2}$ for $x>1$. Its Mellin transform is  $\displaystyle  \frac{1}{(1-s)(2-s)}$ for $1<\sigma<2$. We obtain $\displaystyle  \frac{n}{m}$ for $m\geq1$.
  \item{ Computation of  $ \displaystyle -\frac{1}{2i\pi}\int_{(c)} \frac{m^s}{(1-s)^2}ds$}. We take $f(x)=\displaystyle  \frac{\log x}{x}$ for $0<x<1$ and $0$  for $x\geq1$. Its Mellin transform is $\displaystyle  -\frac{1}{(s-1)^2}$ for $\sigma>1$. Here we obtain $m\log m$ for $m\geq1$.
     \item{Computation of  $ \displaystyle \frac{k}{n}\int_{(c)} \frac{(m\, n)^s ds}{(1-s)^2\,k^s}$}.  As before we find $\displaystyle m\log (\frac{m\,n}{k})$ if $mn\geq k$ and  $0$ if $mn<k$.
  \item{  Computation of $ \displaystyle \frac{n}{2i\pi}\int_{(c)}\frac{m^sds}{s(2-s)j^s},\;j\geq1$}.
  We take $\displaystyle f(x)= -\frac{1}{2} $ for $0<x\leq1$ and $\displaystyle f(x)= -\frac{1}{2x^2}$ for $x>1$. We get $\displaystyle -\frac{n}{2}$ if $j\leq m$ and $\displaystyle -\frac{nm^2}{2j^2} $ if $j>m$.
  
 \item{ Computation of  $ \displaystyle -\frac{1}{2i\pi}\int_{(c)}\frac{m^sds}{s(1-s)j^s}, \;j\geq1$}. Here we  take $\displaystyle f(x)=1-\frac{1}{x}$ if $0<x\leq 1$ and $0$ for $x>1$. 
  We obtain $\displaystyle 1-\frac{m}{j}$ if $m\geq j$ and  $0$ otherwise. 
  \item{ Computation of  $ \displaystyle \frac{k}{n}\frac{1}{2i\pi}\int_{c)}\frac{(nm)^s}{s(1-s)(jk)^s}$}.
 Here we obtain $\displaystyle 1-\frac{nm}{jk}$ if $mn\geq jk$ and  $0$ otherwise.
  \end{enumerate}
  By putting together these partial results we end the proof of Theorem \eqref{extendedFranel}.
 \subsection{Second approach $ \left\{\frac{ \theta}{x}\right\}$}
The most interesting approach for the evaluation of the integral $\displaystyle \int_0^1 \left\{ \frac{ \theta}{t} \right\} \sin (n \pi t) dt $ is to use \eqref{Franel3}:
\[
\int_0^1 \left\{ \frac{ \theta}{t} \right\} \sin (n \pi t) dt = \int_0^\theta \left\{ \frac{ \theta}{t} \right\} \sin (n \pi t) dt + \theta \int_\theta^1 \frac{ 1}{t} \sin (n \pi t) dt.
\]
Moreover
\begin{equation*} \begin{split}
\int_0^\theta \left\{ \frac{ \theta}{t} \right\} \sin (n \pi t) dt &=
 \sum_{p \geq 1} \int_{\frac{\theta}{p+1}}^{ \frac{ \theta}{p}}  \sin (n \pi t) ( \frac{ \theta}{t}-p) dt \\
 & = \theta \sum_{p \geq 1} \int_{\frac{\theta}{p+1}}^{ \frac{ \theta}{p} } \frac{\sin (n \pi t)}{t} dt - \sum_{p \geq 1} p \int_{\frac{\theta}{p+1}}^{ \frac{ \theta}{p}} \sin (n \pi t)dt \\
 &= \theta \int_0^\theta \frac{ \sin (n \pi t)}{t} dt + \frac{1}{n \pi}\sum_{p \geq 1} p \Big( \cos \frac{n\pi \theta}{p} - \cos \frac{ n \pi \theta}{p+1} \Big)  \\
 \end{split} \end{equation*}
 
 Hence  the  $n$-th Fourier coefficient
  \[\displaystyle  a_n= (\{\frac{\theta}{x}\}{\big\vert}\sqrt{2}\sin(n\pi x))= \sqrt{2}\int_0^1\{\frac{\theta}{x}\}\sin(n \pi x) dx\]
is also
 \[
 a_n = \sqrt{2} \left ( \theta \int_0^1 \frac{ \sin (n \pi t)}{t} dt +\frac{1}{n \pi} \sum_{p \geq 1} p \Big( \cos \frac{n\pi \theta}{p} - \cos \frac{ n \pi \theta}{p+1} \Big) \right ) .
 \]
 Seeking for the  coefficient  corresponding to $  f(x) = \sum_{1 \leq \nu \leq N} c_\nu \{ \frac{ \theta_\nu}{x} \}$ the first integral does not matter since  $\sum_{1 \leq \nu \leq N} c_\nu \theta_\nu =0$.
 It remains to compute
 \[
 A =  \sum_{p \geq 1} p \Big( \cos \frac{n\pi \theta}{p} - \cos \frac{ n \pi \theta}{p+1} \Big) .
 \]
 $A$ depends on $n$ and $\theta$. We first consider the finite sum
 \[
  A_N =\sum_{1 \leq p \leq N} p \Big( \cos \frac{n\pi \theta}{p} - \cos \frac{ n \pi \theta}{p+1} \Big) \]
  and set  $x= n\pi \theta$. We have by a partial summation
  \begin{eqnarray*}
 A_N& =& \cos \frac{x}{1} + \cdots + \cos \frac{x}{N} - N \cos \frac{x}{N+1} \\
 &= &(\cos \frac{x}{1}-1) + \cdots + (\cos \frac{x}{N}-1) +N(1- \cos \frac{x}{N+1})\\
 &=& -2 \sum_{k=1}^N \sin^2 \frac{x}{k} +2N \sin^2 \frac{x}{N+1} .
 \end{eqnarray*}
 Hence
 \begin{equation}
 \lim_{N\to + \infty} A_N=  -2 \sum_{k=1}^{\infty} \sin^2 \frac{x}{k}.
 \end{equation}
 We thus obtain one of our main results: the  $n$-th Fourier coefficient $a_n$ of the fundamental function $\{\frac{\theta}{\bullet}  \} $ is  related to the value at $n$ of the antiderivative of the  function $\displaystyle \mathfrak f (x)$ given in \eqref{Fundamental}
 \begin{equation}\label{main}
  a_n = \sqrt{2} \left ( \theta \int_0^1 \frac{ \sin (n \pi t)}{t} dt    -\frac{1}{n \pi}     \sum_{k=1}^{\infty} \sin^2 \frac{n\pi \theta}{k}             \right),
 \end{equation}
bearing in mind that the derivative of $\displaystyle   \sum_{k=1}^{\infty} \sin^2 \frac{u}{k}$ is  $ {\mathfrak f }(2u)   $.
  \medskip
 
To give some useful integral representations we adapt an interesting method, due to Delange \cite{Delange}, and use a result of Saffari and Vaughan \cite{Saffari}. First we introduce
for $ 0<\alpha\leq 1$
\[c_{\alpha}(u)=\left\{ \begin{matrix}
1 & {\rm if}\quad  u-\floor{u}= \rho(u)< \alpha\\
\\
0& \;\;{\rm otherwise}
\end{matrix}
\right.
\]
Furthermore for $x>0,\,y>1$ let 
\[\vartheta_{x,y}(u)= \frac{1}{\log y}\sum_{n\leq y}\frac{1}{n}c_{\alpha}(\frac{x}{n}). \]
According to \cite{Saffari} we have
\begin{lemma}\label{Saffari} We have the estimate
\[\vartheta_{x,y}(u)= u+ O\left((\log x)^{\frac{2}{3}}(\log y)^{-1}  \right),\]
the $O$ is uniform in $u$.
\end{lemma}
 If $f$ is  continuously differentiable function on $ [0, 1] $
\[ f(2\pi \frac{x}{n})= -2\pi \int_{\{\frac{x}{n} \}}^1 f'(2\pi u)\, du = -2\pi\int_0^1 c_u(\frac{x}{n}) f'(2\pi u)\, du.\]
Hence
\begin{equation}\label{Delange}
\sum_{n\leq x} \frac{1}{n}f(2\pi \frac{x}{n})= -2\pi(\log x) \int_0^1 \vartheta_{x,x}(u) f'(2\pi u),
\end{equation}
since
\[ \int_0^1 \vartheta_{x,x}(u) f'(2\pi u)\, du=    \int_0^1  f'(2\pi u)\, du  +     \int_0^1\left( \vartheta_{x,x}(u)-u\right) f'(2\pi u)\, du\] \[=  \int_0^1\left( \vartheta_{x,x}(u)-u\right) f'(2\pi u)\, du. \]
From the Lemma \eqref{Saffari} we get, since $f'$ is bounded on $(0, 1) $ 
\[\sum_{n\leq x} \frac{1}{n}f(2\pi \frac{x}{n})= O\left(\log x \right)^{\frac{2}{3}}.\]
A natural example is to consider a Dirichlet character modulo $N,\, \chi$. In this case
\[\sum_{n\leq x} \frac{\chi(n)}{n}\sin (2\pi \frac{x}{n})= O\left(\log x \right)^{\frac{2}{3}}.\]
We shall not try to give sufficient conditions to justify the process here. The main interest of the remark is that it suggests a method of dealing with various other sums than ${\mathfrak f}(x)$.
 
\section{Almost periodicity}
The goal
 of this section is to show, by elementary methods, that the Hardy-Littelwood-Flett function 
$\displaystyle {\mathfrak f}(x) = \displaystyle \sum_{n=1}^{\infty} \frac{1}{n}  \sin  \frac{x}{n} $ is not bounded on the real line. First we recall two fundamental results on 
Bohr-almost periodic functions \cite{Bohr} (p.39, 44, and 58).
\begin{thm}[The Mean value theorem]
For every almost periodic function $f(x)$, there exists a mean value
\[\lim_{T\to +\infty} \frac{1}{T} \int_0^T f(x)dx= M\{f(x)\}\]
and 
\[\lim_{T\to +\infty} \frac{1}{T} \int_a^{a+T} f(x)dx= M\{f(x)\}\]
uniformly with respect to $a$. In particular if $f$ is an odd almost periodic function, then its mean $M\{f(x)\}$ is zero.
\end{thm}
\begin{thm}[The antiderivative theorem]\label{anti}
The integral $ \displaystyle  F(x)= \int_0^x f(t) dt$ of an almost-periodic function $f(x)$ is  almost-periodic if and only if  it is bounded.
\end{thm}
Let 
\[  \mathfrak{F}(x)= \sum_{n=1}^{\infty} \frac{1}{n^2}\cos \frac{x}{n}. \]
The series defining $ \mathfrak{F}(x)$ is uniformly convergent on the real line. The partial sums
\[ \mathfrak{F}_n(x) =    \sum_{p=1}^{n} \frac{1}{p^2}\cos \frac{x}{p}     \]
are  almost periodic \cite{Bohr} (Corollary, p.38), and then $ \mathfrak{F}(x)$ is also almost periodic \cite{Bohr} (Theorem IV, p.38). It is interesting to note that $\mathfrak{F}_n$  is periodic of period $p_n= \displaystyle \lcm(1,2,\cdots,n)= e^{\psi(n)}             $, with $\psi(x)$ is the Chebyshev function, given by $\displaystyle 
\psi(x)=     \sum_{p\leq x} \Lambda(p)      $, where $\Lambda(n)$ is the Mangoldt function.

The prime number theorem asserts that    $\displaystyle p_n= e^{n(1+o(1))} $ as $n \to \infty$ \cite{Tenenbaum} (p.261). Actually $\displaystyle p_n\leq 3^n$.
 
\begin{lemma}\label{limit} We have
\[ \lim_{x\to +\infty }  \frac{1}{x} \sum_{n=1}^{\infty} \sin^2\frac{x}{n}= \frac{\pi}{2}. \]
\end{lemma}
Let $x>0$ and $\displaystyle n_x= \floor {\frac{2x}{\pi}}  $. The function $\displaystyle h: x\to \sin^2 \frac{1}{x}$, being bounded on $[0, \frac{\pi}{2}]  $ and continuous on each
 $\displaystyle [\alpha, \frac{\pi}{2}    ]    $, is Riemann-integrable on $[0, \frac{\pi}{2}]  $, so by considering Riemann sums
 \begin{equation}\label{Flett1}
 \lim_{x\to +\infty }  \frac{1}{x} \sum_{n=1}^{n_x} \sin^2\frac{x}{n}= \int_0^{ \frac{2}{\pi} }\sin^2 \frac{1}{t}\,dt= \int_{ \frac{\pi}{2} }^{\infty} \frac{\sin^2u}{u^2}\,du.
 \end{equation}
For $x>0$ the function of $g(t)= \sin^2 \frac{x}{t}$ is decreasing on $\displaystyle (\frac{2x}{\pi},\; +\infty)$ and thus
\begin{equation}\label{Flett2}
\left \vert  \sum_{n=n_x+1}^{\infty} \sin^2\frac{x}{n}-\int_{\frac{2x}{\pi}}^{\infty} \sin^2 \frac{x}{t}\,dt     \right \vert \leq 1.
\end{equation}
Since $ \displaystyle   \int_{\frac{2x}{\pi}}^{\infty} \sin^2 \frac{x}{t}\,dt= x\int_0^{\frac{\pi}{2} }  \frac{ \sin^2 u}{u^2}\,du      $ we deduce the lemma from  \eqref{Flett1} \eqref{Flett2} and the relations
\[    \int_0^{\infty }  \frac{ \sin^2 u}{u^2}\,du = \int_0^{\infty }  \frac{ \sin u}{u}\,du= \frac{\pi}{2}.         \]
\begin{corollary}
The function $\displaystyle {\mathfrak f}(x) = \displaystyle \sum_{n=1}^{\infty} \frac{1}{n}  \sin  \frac{x}{n} $  is not bounded on the real line.
\end{corollary}
\begin{proof}
Assume that $\displaystyle {\mathfrak f}(x) $ is bounded on $\R$ then it would be almost periodic by the antiderivative theorem \eqref{anti} and the remark that  $ \displaystyle {\mathfrak f'}(x)= \displaystyle {\mathfrak F}(x) $.
 Since $\displaystyle {\mathfrak f}(x) $ is odd its mean is zero. This is in contradiction with the limit $\displaystyle\frac{\pi}{2} $ given by the Lemma \eqref{limit}.
\end{proof}
\begin{remark}
 The same analysis applies to the series   $\displaystyle \sum_{n=1}^{\infty} \frac{\chi(n)}{n} \sin(\frac{x}{n}),\,\chi$ being a Dirichlet character modulo $N$. 
\end{remark}
%In \cite{Segal} Segal showed that 
%\begin{equation}\label{Segal1}
%\int_0^y{\mathfrak f}(t)\,dt= 2\sum_{d=1}^{\infty} \sin^2\left(\frac{y}{2d} \right)= \frac{\pi y}{2}-\frac{1}{2}+\left( \frac{\pi y}{2}\right)^{\frac{1}{2}} \sum_{d=1}^{\infty} J_1\left(2 \sqrt{2\pi yd} \right)
%\end{equation}

We will need to consider some Bessel functions. We recall that for $ \Re s>0$ the $\Gamma$-function is
 \[ \Gamma(s)= \int_0^{\infty} u^{s-1} e^{-u}\, du.\]
By Fubini's theorem
 \[ \Gamma^2(s)= \int_0^{\infty}  \int_0^{\infty}(uv)^{s-1} e^{-(u+v)}\, dudv= \int_0^{\infty} u^{s-1} \xi_0(u)\,du,\]
 where
 \[\xi_0(u)= \int_1^{\infty}\frac{2e^{2t\sqrt{u}}}{\sqrt{t^2-1}}\,dt.\]
More generally the iterated integrals \cite{Voronoi1}, \cite{Voronoi2}
\[\xi_1(x)= \int_x^{\infty} \xi_0(t)\,dt, \cdots, \xi_k(x)= \int_x^{\infty} \xi_{k-1}(t)\, dt.\]
satisfy the differential equation of Bessel type:
\[x\frac{d^2\xi_k(x)}{dx^2}+(1-k)\frac{d\xi_k(x)}{dx}-\xi_k(x) =0.\]

The ordinary Bessel function of order $\nu$ is 
\[J_{\nu}(z)= \sum_{m=0}^{\infty} \frac{(-1)^m (z/2)^{2m+\nu}}{m! \Gamma(m+1+\nu)},\quad I_{\nu}(z)= i^{-\nu}J_{\nu}(iz),\quad  \vert z\vert <\infty. \]

The $K$-Bessel function of order $\nu$, for $\nu $  not an integer,  is
\[K_{\nu}(z)= \frac{\pi}{2} \frac{I_{-\nu}(z)-  I_{\nu}(z)}{\sin \pi \nu}. \]
 When $\nu$ is an integer we take the limiting value. It could be also defined by
\begin{equation}\label{Bessel1}
K_{\nu}(z)= \frac{1}{2}\int_0^{\infty} t^{\nu-1} e^{-z/2(t+1/t)} dt,\quad \Re \nu \geq 0.
\end{equation}
The Mellin transform of the $J_0$-Bessel function is: \[\int_0^{\infty } J_0(\sqrt{x}) x^{s-1}dx= 4^s \frac{\Gamma(s)}{\Gamma(1-s)}. \]
We will need two Mellin transforms, due essentially to Voronoi
\[
\int_0^{\infty} x^{s-1} K_0(4\pi \sqrt{x})\,dx= \frac{1}{2} (2\pi)^{-2s} \Gamma^2(s), \]
 \[ \int_0^{\infty} x^{s-1} Y_0(4\pi \sqrt{x})\,dx= -\frac{1}{\pi} (2\pi)^{-2s}\cos (\pi s) \Gamma^2(s).
\]

\subsubsection{Summations formulas and beyond}
Various classical summation formulas, as Poisson summation formula, Voronoi summation formula or Hardy-Ramanujan summation formula can all be given a unified formulation.
The following  Generalized Poisson summation formula  is proved in \cite{Cohen}
\begin{thm}
Let $a= a(n)$ be an arithmetic function with moderate growth. We define the Dirichlet series
		\[L(a, s)= \sum_{n=1}^{\infty} a(n)n^{-s},\quad \Re s>1\]
and we suppose that $ L(a, s)$ has an analytic continuation to $\mathbb C$ with only a possible pole
 at $s = 1$. We suppose also that there are  positive constants $A, a_1,\cdots ,a_g $ such that  with the $\Gamma$-factors
\[\gamma(s)=  A^s \prod_{j=1}^g \Gamma(a_j s)\]
$L(a, s)$ satisfies the functional equation \[\gamma(s)L(a, s) = \gamma(1-s)L(a, 1-s).\] Furthermore
for $f\in {\mathcal S}(\mathbb R)$, the Schwartz space, we define a very special Hankel' s transform:
\[g(x)= \int_0^{\infty} f(y)K(xy)dy, \quad {\rm with }\quad     K(x) = \int_{\Re s= \frac{3}{2}} \frac{\gamma(s)}{\gamma(1-s)} x^{-s} ds.\]
 Then, 
\[\sum_{n=1}^{\infty}a(n)f(n)= f(0)L(a, 0) + {\rm Res}_{s=1} {\mathcal M} (f)(s)L(a, s) + \sum_{n=1}^{\infty}a(n)g(n)\]
where $\displaystyle  {\rm Res}_{s=1}     $ is the evaluation of the residue at $s=1$.
\end{thm}
\subsubsection{2 classical choices}
\begin{enumerate}
\item For $ a(n)= 1$ we have $L(a,s)= \zeta(s)$ and 
 \[\gamma(s)=   \pi^{-\frac{s}{2}}  \Gamma(\frac{s}{2}),\quad 
 K(x)= 2\cos (2\pi x).\]
We recover Poisson summation formula for even functions in $f(x)\in {\mathcal S}(\R)$: 
\begin{equation}\label{Poisson} 
\sum_{n=1}^{\infty} f(n)= -\frac{1}{2}f(0)+ \int_0^{\infty} f(x) dx +2 \sum_{n=1}^{\infty}  \int_0^{\infty} f(y) \cos (2\pi ny) dy.
\end{equation}
 \item If $a(n)= d(n)$ we have $ L(d, s)= \zeta^2(s)$ and 
 \[\gamma(s)=  \pi^{-s} \Gamma(\frac{s}{2})^2,\quad
 \frac{\gamma(s)}{\gamma(1-s)}= (2\pi)^{-2s} (2+2\cos \pi s) \Gamma(s)^2\]
 and
 \[K(x)= 4 K_0(4\pi \sqrt{x})- 4 Y_0(4\pi \sqrt{x}).\]
 \end{enumerate}
 We recover Voronoi summation formula 
\begin{align*}
\sum_{n=1}^{\infty}f(n)d(n)& = \frac{1}{4}f(0)+ \int_0^{\infty} f(x)\left(2\gamma+ \log x \right) dx+\\ 
\sum_{n=1}^{\infty} d(n)&  \int_0^{\infty} f(y)
\left(4 K_0 \left(4\pi (ny)^{\frac{1}{2}} \right)-2\pi Y_0 \left(4\pi (ny)^{\frac{1}{2}} \right)\right) dy.
\end{align*}
As a consequence we have  Koshliakov’s formula valid for $a>0$:
\[\sqrt{a}\left(\gamma-\log\left(\frac{4\pi}{a}\right)   +4 \sum_{n=1}^{\infty} d(n) K_0(2\pi an)      \right) \]
\[= \frac{1}{\sqrt{a}}\left(\gamma-\log(4\pi a)  +4 \sum_{n=1}^{\infty} d(n) K_0\left(\frac{2\pi n}{a}\right)      \right).\]
This formula was proved by Ramanujan about ten years earlier (He did not appeal to Voronoi’s formula) and by many authors later.
\subsection{Another function of Hardy and Littlewood}
 Hardy and Littlewood gave in \cite{Hardy} (p.269) the following relation
\begin{align}\label{Hardy2}
{\mathfrak F}(z)= \sum_{n=1}^{\infty} \frac{1}{n}\left(1-e^{-z/n} \right)&= 2\log z+ 2\gamma\\
&-2\sum_{n=1}^{\infty} \left\{ K_0\left(\sqrt{2n\pi iz}\right)+    K_0\left(\sqrt{-2n\pi iz}\right)                   \right\} \nonumber
\end{align}
where $\Re z>0 , \gamma$ is Euler's constant.

For $\vert z \vert< 1$:
\begin{equation}\label{Segal2}
\sum_{n=1}^{\infty} \frac{1}{n} \left(1-e^{-\frac{z}{n}}\right)= -\sum_{n=1}^{\infty}\zeta(n+1) \frac{(-z)^n}{n!}.
\end{equation}
An immediate consequence of this expansion is obtained by taking real and imaginary parts with $z= ix, x\in \R, \vert x\vert<1$:
\begin{eqnarray*}
\sum_{n=1}^{\infty} \frac{1}{n} \left(1-\cos\frac{x}{n}\right)&= \displaystyle2\sum_{n=1}^{\infty} \frac{1}{n} \sin^2\frac{x}{2n}= - \sum_{k\geq 0} \zeta(4k+1)\frac{x^{4k}}{(4k)!}+ \sum_{k\geq 0} \zeta(4k+3)\frac{x^{4k+2}}{(4k+2)!}\\
&= \displaystyle \sum_{k\geq 0}(-1)^{k-1} \zeta(2k+1)\frac{x^{2k}}{(2k)!}
\end{eqnarray*}

\begin{equation*}
\sum_{n=1}^{\infty} \frac{1}{n} \sin\frac{x}{n}= \sum_{k\geq 0} \zeta(4k+2)\frac{x^{4k+1}}{(4k+1)!}-  \sum_{k\geq 0} \zeta(4k+4)\frac{x^{4k+3}}{(4k+3)!}.
\end{equation*}
More generally we define the series
 \[G_{\nu}(z)= \sum_{n> \Re \nu+1} \zeta(n-\nu) \frac{(-z)^n}{n!}\]
which has   a Mellin-Barnes type integral representation when  $x>0, c$ is fixed with $\Re \nu +1<c< \Re \nu  +2 $:
\begin{equation*}
G_{\nu}(z)= \frac{1}{2i\pi} \int_{(c)} \Gamma(-s) \zeta(s-\nu) x^s ds.
\end{equation*}
The proof of the main equality results from the deformation of the path of integration and the fact that the pair
\[x^{\nu}K_{\nu} (x),\quad 2^{s+\nu-2}\Gamma(s/2) \Gamma(s/2+\nu),\quad \Re s>  {\max}(0, -2\Re \nu) \]
is a pair of Mellin transforms \cite{Katsurada1}, \cite{Katsurada2}.\\
The series \eqref{Segal2} has many remarkable properties. It may be differentiated term by term to get
$    {\mathfrak G}(-x)     $ where $    {\mathfrak G}(x)$ is the function  defined in \cite{Tenenbaum} (p.243):
\begin{equation}\label{Segal3}
 {\mathfrak G}(x)= \sum_{n=1}^{\infty} \frac{1}{n^2} e^{\frac{x}{n}}.
\end{equation}
The following formula is mentioned in \cite{Tennenbaum} 
\[\lim_{K\to \infty}\frac{1}{K} \sum_{k=1}^K{\mathfrak G}(2i\pi nk)= \sum_{d\mid n} \frac{1}{d^2}.\]
\subsection{ Laplace transform of $a \sqrt{t} J_1(a\sqrt{t}),\; a>0,\; t>0$, another approach to Segal's formula}
In \cite{Segal} (formula (12)) Segal proves the following result
\begin{thm}
If $\displaystyle g(z):= \sum_{k \geq 1} (1- \cos \frac{z}{k})$ then
\[g(z)= \frac{\pi z}{2}- \frac{1}{2}+ \frac{1}{4} \sum_{k \geq 1} \frac{2 \sqrt{2k\pi} \sqrt{z}}{k} J_1(2 \sqrt{2k\pi} \sqrt{z}).\]
\end{thm}
\noindent This formula is interesting compared to \eqref{Hardy2}, as we have for real $z,\, g(z)= \Re{\mathfrak F}(iz) $.  
 The proof given in \cite{Segal}  uses a rather elaborated tools such the three Bessel functions $J_1, J_2, J_3$, the functional equation of the Riemann $\zeta$-function etc.  We give here a  proof which we think is simpler.
 
Let \[g_1(z)= \frac{\pi z}{2}- \frac{1}{2}+ \frac{1}{4} \sum_{k \geq 1} \frac{2 \sqrt{2k\pi} \sqrt{z}}{k} J_1(2 \sqrt{2k\pi} \sqrt{z}).\]

  In the Laplace transform 
\[
{\mathcal L}(a \sqrt{t} J_1(a\sqrt{t}))(p) = a \int_0^{+ \infty} \sqrt{t} J_1(a\sqrt{t}) e^{-tp} dt ,\quad \Re\, p >0
\]
we set $u^2=t$ and obtain
\[
{\mathcal L}(a \sqrt{t} J_1(a\sqrt{t}))(p) = 2a \int_0^{+ \infty}  J_1(a u) e^{-p u^2} u^2 du ,\quad \Re\, p >0.
\]
According to \cite{Watson} (page 394, formula(4)) we have for with $\vert \Arg \; p \vert < \frac{\pi}{4} $
\[
\int_0^{+ \infty} J_\nu(au) e^{-p^2 u^2} u^{\nu+1} du = \frac{a^\nu}{(2p^2)^{\nu+1}} e^{ - \frac{a^2}{4p^ 2}}.
\]
Replacing   $p$ by $\sqrt{p}$ with  $\vert \Arg \, p \vert < \frac{\pi}{2}$ and  taking $\nu= 1 $ we obtain
\[\displaystyle
\int_0^{+ \infty} J_1(au) e^{-p u^2} u^2 du = \frac{a}{4p^2}e^{- \frac{a^2}{4p}}.
\]
Hence 

\[
{\mathcal L}(a \sqrt{t} J_1(a \sqrt{t} ))(p) = \frac{a^2}{2 p^2} e^{ \frac{-a^2}{4 p}} \quad \Re p>0.
\]
Note that $a \sqrt{t} J_1(a\sqrt{t}) $ is not in $L^2([0,+ \infty[)$ since its Laplace transform is not bounded in the $L^2$-norm  on the lines $\Re p= c$. With $a=2 \sqrt{2k\pi} $ we get
\[
{\mathcal L}(2 \sqrt{2 k \pi t} J_1(2 \sqrt{2k\pi t} )) (p) = \frac{4 k \pi}{p^2} e^{-\frac{2k\pi}{p}}. 
\]
As we have
\[
{\mathcal L}( \frac{\pi t-1}{2} )(p) = \frac{\pi}{2 p^2}- \frac{1}{2p}
\]

and, by continuity, the Laplace transform of the sum in
\[
g_1(t) = \frac{\pi t-1}{2}+ \frac{1}{4} \sum_{k \geq 1} \frac{2 \sqrt{2k\pi t} }{k} J_1(2 \sqrt{2 k \pi t})
\]
 is
\[
\frac{\pi}{p^2} \sum_{k \geq 1} (e^{-\frac{2 \pi}{p}})^k
\]
which converges in  $\Re p>0$, the Laplace transform of  $g_1(t)$ is
\[
\frac{\pi}{2p^2} -\frac{1}{p}+ \frac{\pi}{p^2} \frac{e^{-\frac{2 \pi}{p}}}{1-e^{-\frac{2 \pi}{p}}}=\frac{\pi}{2p^2} -\frac{1}{p}+ \frac{\pi}{p^2} \frac{1}{e^{\frac{2 \pi}{p}}-1}
.\]
On the other hand
\[{\mathcal L}(1-\cos\frac{t}{k})(p)= \frac{1}{p}-\frac{p}{p^2+k^{-2}}= \frac{1}{p(p^2k^2+1)}                                  \]
and
\[ {\mathcal L}(g) (p)= \sum_{k=1}^{\infty}     \frac{1}{p(p^2k^2+1)}.    \]
The equality $\displaystyle  {\mathcal L}(g) (p)= {\mathcal L}(g_1) (p)    $  is obtained   by using the well known  partial fractions decomposition
\[
\frac{z}{e^z-1} =1 - \frac{z}{2} + \sum_{k \geq 1} \frac{2z^2}{z^2+ 4 k^2 \pi^2 },\quad z \in \C \setminus 2i\pi \Z,
\]
where we have to set $\displaystyle z= \frac{2\pi}{ p}$. Hence $g=g_1$  by injectivity of Laplace Transform.
\subsection{Some Mellin transforms and the cube of theta functions}
It has been noticed in \cite{Flett} (p.14) that the function
\[R(t)=  \sum_{n\leq t} \frac{1}{n} e^{\frac{it}{n}}  \]
is very similar to $\zeta(1+it)$ in its asymptotic behaviour as $t\to +\infty$. This could suggest a link between this function and the theta series
 $\displaystyle \vartheta_3(q)= \sum_{n\in \BZ}q^{n^2}$. 
In this section, following a suggestion of Crandall \cite{Crandall}, we would like to briefly show by considering  Mellin transforms  an unexpected link to the third power of the (fourth) Jacobi theta function  $\displaystyle \vartheta_4(q)= 
\sum_{n\in \Z} (-1)^n q^{n^2},\; \vert q\vert<1$. We define 
\begin{eqnarray}
 {\tilde \chi}(s, t)&=& \sum_{n=1}^{\infty}  \frac{e^{-\frac{t}{n}}}{n^s} \label{Flett1},\\
  \chi(s, t)&=& \sum_{n=1}^{\infty} (-1)^n \frac{e^{-\frac{t}{n}}}{n^s}. \label{Flett2}
\end{eqnarray}
These two functions are defined for $s\in \C$ and   $\Re s>1 $ for $ {\tilde \chi}(s, t)$, $\Re s >0 $ for $ \chi(s, t)$. They are related by
\[   \chi(s, t)= \frac{1}{2^{s-1}}  {\tilde \chi}(\frac{s}{2}, t)-  {\tilde \chi}(s,  t).\]
We have
\begin{equation*}
\int_0^{\infty} t^{s-1}  {\tilde \chi}^2(s, t)\, dt= \Gamma(s) \sum_{n,m= 1}^{\infty} \frac{1}{(n+m)^s}= \Gamma(s) \left(\zeta(s-1)-\zeta(s) \right)
\end{equation*}
\begin{equation*}
\int_0^{\infty} t^{s-1}  \chi^2(s, t)\, dt= \Gamma(s) \sum_{n,m= 1}^{\infty} \frac{(-1)^{n+m}}{(n+m)^s}= \Gamma(s) \left(\eta(s-1)-\eta(s) \right),
\end{equation*}
where
\[ \eta(s)= \sum_{n=0}^{\infty} \frac{(-1)^n}{(2n+1)^s}\]
is the Dirichlet $\eta$-function. Furthermore
\begin{equation*}
 \frac{1}{\Gamma(s) }\int_0^{\infty} t^{s-1}  {\tilde \chi}^3(s, t)\, dt= \sum_{p,q,r= 1}^{\infty} \frac{1}{(pq+qr+rp)^s}
\end{equation*}
and for $ \chi(s, t)$ we have a more interesting result 
\begin{equation}\label{Crandall11}
\frac{1}{\Gamma(s) }\int_0^{\infty} t^{s-1} \chi^3(s, t)\, dt= \sum_{p,q,r= 1}^{\infty} \frac{(-1)^{p+q+r}}{(pq+qr+rp)^s}.
\end{equation}
We have the following lemma due to Andrews \cite{Andrew} (p.124)
\begin{lemma} For  $ \vert q\vert<1$
\begin{equation}\label{Crandall22}
\vartheta_4^3(q)= \sum_{n\in \Z}(-1)^n r_3(n)q^{n}= 1+4\sum_{n=1}^{\infty} \frac{(-1)^n q^n}{1+q^n}
-2 \sum_{\substack{n=1 \\ \vert j\vert \leq n}}^{\infty}
\frac{q^{n^2-j^2}(1-q^n)(-1)^j}        {1+q^n}\end{equation}
\end{lemma}
\noindent where $r_3(n)$ is the number of representations of $n$ as sum of three squares. According to a result of Fermat an integer is a sum of three squares
 if and only if  it is not of form $4^n(8m+7)  $. There are some gaps in the expansion in power series of the left hand side of \eqref{Crandall22}.
Similarly to \eqref{Crandall11} we have
\begin{equation}\label{Crandall 3}
    \frac{1}{\Gamma(s) }\int_0^{\infty}t^{s-1}\left(\vartheta_4^3(q)-1 \right)\,dt= \sum_{p,q,r\in \BZ}^{\infty}{\vphantom{\sum}}' \frac{(-1)^{p+q+r}}{(p^2+q^2+r^2)^s}.
\end{equation}
A link between \eqref{Crandall11} and \eqref{Crandall 3} is given by Crandall \cite{Crandall} (p.372)  as a relation between two Epstein zeta functions associated with the two not equivalent ternary forms
\[q_1(u,v,w)= u^2+v^2+w^2,\quad  q_2(u,v,w)= uv+vw+wu \]
in the form
\begin{equation}\label{Crandall4}
\sum_{p,q,r\in \BZ}^{\infty}{\vphantom{\sum}}' \frac{(-1)^{p+q+r}}{(p^2+q^2+r^2)^s}= -6(1-2^{1-s})^2 \zeta^2(s) -4 \sum_{p,q,r= 1}^{\infty} \frac{(-1)^{p+q+r}}{(pq+qr+rp)^s}.
\end{equation}
Next we establish a functional equation
\begin{thm}  For $t>0$ the function  $\displaystyle \chi(\frac{1}{2}, t)$ satisfies the following functional equation
\begin{equation}\label{Crandall1}
\chi(\frac{1}{2}, t)= \sum_{n=1}^{\infty} \frac{(-1)^n}{\sqrt{n}}e^{-t/n}=  \sqrt{i}\, \sum_{\mathcal O}  \frac{e^{-\gamma \sqrt{2\pi dt}}}{\sqrt d},
\end{equation}
with $\gamma= 1-i $  and  $\mathcal{ O}$ is the set of odd integers.
\end{thm}
We could also seek for a result similar to \eqref{Crandall1} for 
\begin{equation}\label{Crandall2}
 {\tilde \chi}(s, t)= \sum_{n=1}^{\infty} \frac{1}{n^s}e^{-t/n},\quad \Re s>1.
\end{equation}
 The relevance of  this function lies in  its relation to a Hardy-Littlewood-Flett like function:
\[ \Im{\tilde \chi}(1, -it)=  \Im \sum_{n=1}^{\infty} \frac{1}{n}e^{it/n}= \sum_{n=1}^{\infty} \frac{1}{n} \sin(\frac{t}{n}).
 \]
In order to prove \eqref{Crandall1} we consider, for a fixed $t>0$,  the function
\[  f(x)= e^{i\pi x} \frac{ e^{-t/x}}{\sqrt x}, \;  x>0 \]
extended to the origin by $f(0)= 0 $ and to $\BR$ as an even function. The obtained function is  $\mathcal C^{\infty}$ on the real line to which we 
apply the Poisson summation formula \eqref{Poisson} to obtain 
\begin{equation}\label{Poisson2} 
\sum_{n=1}^{\infty} \frac{(-1)^n}{\sqrt{n}}e^{-t/n}=  \int_0^{\infty} e^{i\pi x} \frac{ e^{-t/x}}{\sqrt x} dx +2 \sum_{n=1}^{\infty}  \int_0^{\infty} e^{i\pi x} \frac{ e^{-t/x}}{\sqrt x}\cos (2\pi nx) dx.
\end{equation}

\begin{rem}
The function $\displaystyle y \mapsto   \frac{ e^{\frac{-t}{y}}}{{\sqrt y} }$ is  continuous on  $[0, \infty[ $ and decreases to $0$ at infinity, hence the proper
integrals
 $\displaystyle \int_0^{\infty}  f(y)  \cos (2\pi ny) dy, n\geq 0$ are convergent.
\end{rem}
We compute ${\mathcal F}(f)(n)$ as follows
\begin{equation*}
 {\mathcal F}(f)(n)=     \int_0^{\infty}    e^{i\pi y}  \frac{ e^{\frac{-t}{y}}}{{\sqrt y} } \cos(2n\pi y)\, dy 
 \end{equation*}
 \begin{equation}\label{Crandall3}
 = \frac{1}{2}
 \left\{ \int_0^{\infty }    e^{(i\pi+ 2i\pi n) y-t/y}      y^{-1/2}     \,dy  +     \int_0^{\infty }    e^{(i\pi- 2i\pi n) y-t/y}      y^{-1/2}     \,dy              \right\}.
 \end{equation}
We recall the modified Bessel function \eqref{Bessel1}, written in the form
\[ \int_0^{\infty} w^{\nu-1} e^{- w-a/w} \,dw= 2 \left(\frac{1}{a}\right)^{\nu/2} K_{\nu}\left(2\sqrt{a} \right)\]
that we use in the form
\begin{equation}\label{Crandall4}
\int_0^{\infty} w^{\nu-1} e^{-bw-a/w} \,dw= 2 \left(\frac{a}{b}\right)^{\nu/2} K_{\nu}\left(2\sqrt{ab} \right).
\end{equation}
Actually for  \eqref{Crandall3} we need only the simplest case of
\begin{equation}\label{Crandall40}
K_{\frac{1}{2}} (z)= \sqrt{\frac{\pi}{2z}}\, e^{-z}.
\end{equation}
The first integral in the left hand side of \eqref{Crandall3}, with
 \[a= t,\quad  -b= i\pi +2i\pi n= i\pi(2n+1)       \]
is equal to
\begin{equation}\label{Crandall5}
2\left( \frac{t}{-i\pi(2n+1)} \right)^{1/4} K_{1/2}\left(2\sqrt{t\left( -i\pi(2n+1)\right )} \right).
\end{equation}
The second integral with 
 \[a= t,\quad  -b= i\pi -2i\pi n= i\pi(-2n+1)       \]
is equal to
\begin{equation}\label{Crandall6}
2\left( \frac{t}{-i\pi(-2n+1)} \right)^{1/4} K_{1/2}\left(2\sqrt{t\left( -i\pi(-2n+1)\right )} \right).
\end{equation}
By using \eqref{Crandall40} we see that 
  \eqref{Crandall3} is the sum of
   \[  \left( \frac{t}{-i\pi(2n+1)} \right)^{1/4}\sqrt{\frac {\pi}{4\sqrt{t\left( -i\pi(2n+1)\right )} }}  e^{-2\sqrt{t\left( -i\pi(2n+1)\right )} } \]  
   and
   \[      \left( \frac{t}{-i\pi(-2n+1)} \right)^{1/4}\sqrt{\frac {\pi}{4\sqrt{t\left( -i\pi(-2n+1)\right )} }}  e^{-2\sqrt{t\left( -i\pi(-2n+1)\right )} }.               \]
  As in \cite{Crandall} we denote by $\gamma= 1-i, \,d=  \pm 2n+1, \, n\in \BN^{\star}$ with $\sqrt{-\vert d\vert}  =i   \sqrt{\vert d\vert} $.  Then $d$  describes $\mathcal O\setminus\{1\}$ and
\[ 
   \left( \frac{t}{-i\pi d} \right)^{1/4}\sqrt{\frac {\pi}{4\sqrt{t\left( -i\pi d\right )} }}  e^{-2\sqrt{-it\pi d} }
   = \frac{1}{2} \sqrt{i}\;  \frac{e^{-\gamma \sqrt{t\pi d} }}{\sqrt d}.
\]
Hence
\begin{equation}\label{Crandall7}
2 \sum_{n=1}^{\infty}  \int_0^{\infty} f(y) \cos (2\pi ny) dy=        \sqrt{i}\, \sum_{d\in {\mathcal O},\, d\neq 1}      \frac{e^{-\gamma \sqrt{t\pi d} }}{\sqrt d }.     
\end{equation}
For the remaining term in \eqref{Poisson2} we use  \eqref{Crandall5}, with $n= 0$,  to obtain
\[ \int_0^{\infty} f(x)dx= \, {\mathcal F}(f)(0)= \sqrt{i}\,  e^{-\gamma\sqrt{2\pi t}} \]
which together with  \eqref{Crandall7} gives finally  \eqref{Crandall1}:
\begin{equation*}
\chi(\frac{1}{2}, t)= \sqrt{i}\, \sum_{\mathcal O}  \frac{e^{-\gamma \sqrt{2\pi dt}}}{\sqrt d}.
\end{equation*}
In close analogy to Jacobi's transformation of Theta  functions \eqref{Crandall1} appears as a convergence acceleration of a slowly convergent series. \\Incidentally $\displaystyle   \chi (\frac{1}{2}; \frac{t^2}{4}) $ is a Fourier transform of a function of the Schwartz class. Indeed let $\displaystyle g(x)= \frac{1}{1+e^{x^2}}$, the reciprocity formulas are
 \[{\hat g}(t)= \frac{1}{2\pi}  \int_{\BR}  \frac{1}{1+e^{x^2}}e^{i t x} \,dx, \quad  g(x)=  \int_{\BR} {\hat g}(t)e^{-itx}\,dt\]
 with
\[{\hat g}(t)= -\frac{1}{2\pi} \sum_{n >0} (-1)^n \int_{\BR} e^{i t x} e^{-n x^2} dx\]
\[ = 
 -\frac{1}{2\sqrt{\pi}}  \sum_{n>0} \frac{(-1)^n}{\sqrt{n} }e^{- \frac{t^2}{4n} }=  -\frac{1}{2\sqrt{\pi}} \chi (\frac{1}{2}; \frac{t^2}{4} ).    \]

\begin{rem}  The convolution of three functions $f, g, h\in {\mathcal S}(\BR)$ is, as well known,
\begin{eqnarray*}
\left(f\star \left(g\star h\right)\right)(x)&=& \int_{\BR} f(y) (g\star h)(x-y) dy\\
&=& \int_{\BR}f(y) \left( \int_{\BR} g(z)h(x-y-z)\right) dy\\
&=& \int \int_{\BR\times \BR} f(y)g(z)h(x-y-z)dz dy.
\end{eqnarray*}
With $\displaystyle f(x)= g(x)= h(x)= \frac{1}{1+e^{x^2}}$   we have 
\[\left(f\star \left(g\star h\right)\right)(x)= \int \int_{\BR^2} 
   \frac{ dt\,du}{   (1+e^{y^2})             (1+e^{z^2})(1+e^{(x-y-z)^2})} \]
\end{rem} 
so that for the Fourier transform
\[\widehat{f\star g\star h}(t)= (2\pi)^2 {\hat{g}}^3(t)= (2\pi)^2  \frac{1}{8\pi \sqrt{\pi}} \chi^3 (\frac{1}{2}; \frac{t^2}{4})= \frac{\sqrt{\pi}}{2} \chi^3 (\frac{1}{2}; \frac{t^2}{4}) \]
or
\[
 f\star g\star h(x)=  \frac{\sqrt{\pi}}{2}\int_{\BR}  \chi^3 (\frac{1}{2}; \frac{t^2}{4} ) e^{-itx}\,dt.
\]
Evaluating at $x= 0$ we obtain 
\[\int \int_{\BR^2} 
   \frac{ dy\,dz}{   (1+e^{y^2})             (1+e^{z^2})(1+e^{(-y-z)^2})}=  \int \int_{\BR^2} 
   \frac{ dy\,dz}{   (1+e^{y^2})             (1+e^{z^2})(1+e^{(y-z)^2})}\]
    \[=   \frac{\sqrt{\pi}}{2}  \int_{\BR}  \chi^3 (\frac{1}{2}; \frac{t^2}{4} )\,dt =  \sqrt{\pi}  \int_0^{\infty}\frac{1}{\sqrt{u}}  \chi^3 (\frac{1}{2}; u)\,du. \]
    From \eqref{Crandall11}, with $s=\frac{1}{2}$, we have (Compare with \cite{Crandall})
    \begin{equation}\label{Crandall8}
    \sum_{p,q,r= 1}^{\infty} \frac{(-1)^{p+q+r}}{(pq+qr+rp)^{\frac{1}{2}}}=  -\frac{1}{\pi} \int \int_{\BR^2} 
   \frac{ dy\,dz}{   (1+e^{y^2})             (1+e^{z^2})(1+e^{(y-z)^2})}.
   \end{equation}
We end this study by using an interesting integral representation due to Mellin \cite{Mellin} (p. 22-23):
\begin{align}
\frac{\Gamma(s)}{(w_0+w_1+\cdots+w_q)^s}=& \frac{1}{(2i\pi)^q}\int_{\kappa_1-i\infty}^{\kappa_1+i\infty} \cdots \int_{\kappa_q-i\infty}^{\kappa_q+i\infty} 
\frac{\Gamma(s-z_1\cdots z_q)}{w_0^{s-z_1\cdots z_q}}\\
 &\frac{\Gamma(z_1)\cdots \Gamma(z_q)} {w_1^{z_1}\cdots w_q^{z_q} } dz_1\cdots dz_q, \nonumber\\
 \kappa_{\nu}>0,\,\nu=1, \cdots,q ;&\quad  \Re s> \kappa_1+\cdots+\kappa_q>0. \nonumber
\end{align}
In the case of $q=2$ we obtain at once, as in \eqref{Crandall8}
\begin{align*}
 \sum_{p,q,r\geq 1} \frac{(-1)^{p+q+r}}{(pq+qr+rp)^s}=&  - \frac{1}{4\pi^2\, \Gamma(s)} \int_{\kappa_1-i\infty}^{\kappa_1+i\infty} \int_{\kappa_2-i\infty}^{\kappa_2+i\infty} 
\Gamma(s-u_1u_2) \Gamma(u_1)\Gamma(u_2)\\
 &\sum_{p\geq 1}\frac{(-1)^p}{p^{s+u_1-u_1u_2}} \sum_{q\geq 1}\frac{(-1)^q}{p^{s+u_2-u_1u_2}}
\sum_{r\geq 1}\frac{(-1)^r}{r^{u_1+u_2}}\,du_1 du_2\\
= - \frac{1}{4\pi^2} \int_{\kappa_1-i\infty}^{\kappa_1+i\infty} \int_{\kappa_2-i\infty}^{\kappa_2+i\infty} &
\Gamma(s-u_1u_2) \Gamma(u_1)\Gamma(u_2) K(s;u_1,u_2)\,du_1 du_2,
\end{align*}  
where
\[K(s;u_1,u_2)=  \eta(s+u_1-u_1u_2) \eta(s+u_2-u_1u_2) \eta(u_1+u_2).\]
This representation of the Epstein zeta function of the ternary form $q_2(u,v,w)= uv+vw+wu $ in terms of the Dirichlet $\eta$-function and similar other representations can shed some light on their analytic continuation.\\


\begin{thebibliography}{999}
\bibitem{Bod1}
 J. Alcantara-Bode,  An integral equation formulation of the Riemann Hypothesis. Integr. Equations Oper, Theory, 17 (2) (1993), 151-168.

\bibitem{Bod}
 J. Alcantara-Bode, {\em Some properties of the Beurling function}, Pro-Mathematica, XIV, No. 27-28 (2000), 1-11.
 
\bibitem{Andrew} 
G. E. Andrew, {\em The fifth and seventh order mock theta functions}, Trans. Amer. Math. Soc.293 (1986), no. 1, 113-134.

\bibitem{Duarte}
L. B\'aez-Duarte, {\em A strengthening of the Nyman–Beurling criterion for the Riemann hypothesis}, Atti Accad. Naz. Lincei Cl. Sci. Fis. Mat. Natur. Rend. Lincei (9) Mat. Appl. 14 (1) (2003) 5-11.

\bibitem{Balazard1}  
M. Balazard, E. Saias, {\em Notes sur la fonction $\zeta$ de Riemann}, 1, Adv. Math. 139 (1998) 310-321.

\bibitem{Balazard2}  
M. Balazard, E. Saias, {\em Notes sur la fonction $\zeta$ de Riemann}, 4, Adv. Math. 188 (2004) 69-86.

\bibitem{Beurling}
A. Beurling, {\em A closure problem related to the Riemann zeta-function}, Proc. Natl. Acad.
Sci.41 (1955) 312-314.

\bibitem{Boersma}
J. Boersma, J., R.J. Chapman, {\em Solution to problem 94-12(by M. L. Glasser): A nonharmonic trigonometric series}. SIAM Review, 37(3), (1995) 443-448.
\bibitem{Bohr}
 H. Bohr, {\em Almost Periodic Functions}. Julius Springer, Berlin (1933) (Chelsea Publishing Company, N.Y., 1947).
 
 \bibitem{Cohen}
 H. Cohen, {\em Number theory, volume II: Analytic and modern tools}, Graduate Texts in Mathematics, vol. 240, Springer, 2007.
 
 \bibitem{Crandall}
 R. E. Crandall, {\em New representations for the Madelung constant},  Experimental
Mathematics, 8 (1999) 367-379.
 
 
 \bibitem{Davenport1}
 H. Davenport, {\em On some infinite series involving arithmetical functions}, Quarterly J. Math (2) 8 (1937), 8-13.
 
  \bibitem{Davenport2}
  H. Davenport, {\em On some infinite series involving arithmetical functions (II)}, Quarterly J. Math(2)8(1937), 313-320.
  
  \bibitem{Delange}
  H. Delange, {\em Sur la fonction $\displaystyle f(x)= \sum 1/n \sin(x/n)$}. Th\'eorie analytique et \'el\'ementaire des nombres, Caen, 29-30, Septembre 1980, Journ\'ees math\'ematiques SMF-CNRS, 1980.
  
\bibitem{Donoghue}
W. F. Donoghue, {\em Distributions and Fourier Transforms}, Academic, N.Y., 1969.

\bibitem{Flett}
T.M. Flett,  {\em On the function $\displaystyle \sum (1/n) \sin(t/n)$}. J. London Math. Soc., 25:5-9, 1950.

\bibitem{Hardy}
G.H. Hardy and J.E. Littlewood, {\em Notes on the theory of series (XX): on Lambert series}. Proc. London Math. Soc., (2) 41 (1936), 257-270.


\bibitem{Hartman}
P. Hartman, {\em On a class of arithmetical Fourier series}, Amer. J. Math. 60 (1938), 66-74.

\bibitem{Katsurada1}
M. Katsurada, {\em On Mellin-Barnes Type of Integrals and Sums associated With the Riemann Zeta function}. Publ. Inst. Math.(Beograd) (Nouvelle Ser.) 62 (76), (1997) 13-25.


\bibitem{Katsurada2}
M. Katsurada, {\em Power series with the Riemann zeta-function in the coefficients}, Proc. Japan Acad. 72 (1996), 61-63.


\bibitem{Mellin}
H. Mellin, {\em Eine Formel fur den Logarithmus transcendenter Funktionen von endlichen Geschlecht}, Acta Sot. Sci. Fennicae 29 (1900) 3-49.

\bibitem{Mordell}
 L. J. Mordell, {\em Integral formulae of arithmetical character}, Journal London Math. Soc, 33 (1958), 371-375.

\bibitem{Nikolski1} 
 N. Nikolski, {\em Distance formulae and invariant subspaces, with an application to localization of zeros of the Riemann $\zeta$-function}. Ann. Inst. Fourier (Grenoble), 45 (1)(1995), 143-159.
 
\bibitem{Nikolski2}
N. Nikolski,  {\em Hardy Spaces} (Cambridge Studies in Advanced Mathematics). Cambridge: Cambridge University Press.(2019).  

\bibitem{Saffari}
 B. Safari, R.C. Vaughan,  {\em On the fractional parts of  $\displaystyle \frac{x}{n} $ and related sequences II}, Ann. Inst. Fourier (Grenoble) 27, 2 (1977), 1-30.

\bibitem{Segal}
S.L. Segal,  {\em On $\displaystyle f(x)= \sum 1/n \sin(x/n)$},      J. London Math. Soc., (2)4, 385-393(1972).


\bibitem{Tenenbaum}
G. Tenenbaum, {\em Introduction to analytic and probabilistic number theory}, volume 163, Graduate Studies
in Mathematics. American Mathematical Society, Providence, RI, third edition, 2015. Translated from the
2008 French edition by Patrick D. F. Ion

\bibitem{Tennenbaum}
J. Tennenbaum, {\em On the function $\displaystyle \sum_{n=0}^{\infty}\frac{\zeta(n+2)}{n!}x^n  $}, Math. Scand. 41 (1977), 242-248.

\bibitem{Titch}
E.C. Titchmarsh, {\em The theory of the Riemann zeta-function}, Clarendon Press, Oxford, 1951

\bibitem {Vasyunin} 
V.I.  Vasyunin,  {\em On  a  biorthogonal  system  related  with  the  Riemann  hypothesis},  St.  Petersburg  Math.  J.  7  (1996) 405-419.

\bibitem{Voronoi1}
G. Voronoi, {\em Sur une fonction transcendante et ses applications \`a la sommation de quelques s\'eries}, Annales Scientifiques de l’\'Ecole Normale Sup\'erieure, troisi\`eme s\'erie 21 (1904), 207-267. 

\bibitem{Voronoi2}
   G. Voronoi, {\em Sur une fonction transcendante et ses applications \`a la sommation de quelques s\'eries, seconde partie}, Annales Scientifiques de l’\'Ecole Normale Sup\'erieure, troisi\`eme s\'erie 21 (1904), 459-533.
 
 \bibitem{Walfisz}  
 A. Walfisz, {\em  Teilerprobleme}, Math. Z. 26 (1927) 66-88.
\bibitem{Watson}
G.N. Watson, {\em  A Treatise on the Theory of Bessel Functions},  Cambridge University Press (1922)

\end{thebibliography}
\end{document}